\documentclass[12pt,reqno]{article}
\usepackage{enumerate}
\usepackage{fullpage,hyperref,bm}
\usepackage{amsmath, amssymb, amsthm, url, graphicx,rotating,tikz,tkz-graph,relsize}
\usepackage{relsize}
\usepackage{mathrsfs}
\usepackage{dsfont}
\usepackage{ulem}
\usepackage{titlesec}

\usepackage[noabbrev]{cleveref}

\hypersetup{
    colorlinks,
    linkcolor={red!50!black},
    citecolor={blue!50!black},
    urlcolor={blue!80!black}
}

\titleformat{\section}[hang]{\centering \scshape}{\thesection}{1em}{}

\newtheorem{thm}{Theorem}
\newtheorem{lm}[thm]{Lemma}
\newtheorem{crl}[thm]{Corollary}
\newtheorem{prop}[thm]{Proposition}

\newcommand{\ZZ}{\mathbb Z}
\newcommand{\one}{\mathbf 1}
\newcommand{\SA}{\mathsf{SA}}
\newcommand{\SAN}{\mathsf{SAN}}
\newcommand{\CA}{\mathsf{CA}}
\newcommand{\CAN}{\mathsf{CAN}}

\newcommand{\sett}[2]{ \left\{ #1 \, \, || \, \, #2 \right \} }

\begin{document}
\thispagestyle{empty}
\setcounter{page}{1}
\title{Skirting the $n$-tuples}
\author{
Sam Adriaensen
 \thanks{Department of Mathematics and Data 
Science, Vrije Universiteit Brussel, Pleinlaan 2, 
1050 Elsene, Belgium. The author is supported by post-doctoral fellowship 12A3Y25N by the Research Foundation Flanders (FWO). {\tt sam.adriaensen@vub.be}}
 \and
Ferdinand Ihringer
 \thanks{Department of Mathematics, Southern 
University of Science and Technology, Shenzhen, 
Guangdong, China. {\tt ihringer@sustech.edu.cn} }
 \and
William J.~Martin
 \thanks{Department of Mathematical Sciences,
Worcester Polytechnic Institute, Worcester, MA USA. {\tt  martin@wpi.edu}, {\tt  rvillagran@wpi.edu} }
 \and
Ralihe R.~Villagr\'an
 \footnotemark[3]
 }

\date{\today} 
\maketitle

\medskip

\begin{abstract}
Let $n\ge 2$ and $q\ge 2$ be given. The set $X = \ZZ_q^n$ is a metric space 
of diameter $n$ under the Hamming metric $d(\cdot,\cdot)$.  We seek a smallest 
set $S\subseteq X$ that ``skirts'' every $q$-ary $n$-tuple in the sense that 
every $x\in X$ is at distance $n$ from at least one element of $S$. Thus we aim 
to compute the total domination number $f(n,q)$ of the graph $G(n,q)$  with vertex set 
$X$ and edge set $\{ xy \, \| \, d(x,y)=n\}$. We provide constructions and bounds for this 
number, establishing $f(n,q) = C_q^{(1+o(1))n}$ for some constants $2=C_2>C_3 \geq \cdots$
which we are only able to estimate at the present time.
\end{abstract}

\noindent {\bf Keywords:} Hamming distance; covering array; total domination.

\noindent {\bf 2020 MSC Subject Codes:} 05B40, 05B15, 05B99, 05C69, 94B65.

%
%
\section{Introduction}
\label{Sec:intro}

Consider an integer $q \geq 2$, and the ring $\ZZ_q = \ZZ / q \ZZ$.
We refer to the elements of $\ZZ_q^n$ as \underline{$q$-ary $n$-tuples}.
The \underline{Hamming distance} between two $q$-ary $n$-tuples $x$ and $y$, denoted $d(x,y)$, equals the number of positions where $x$ and $y$ have a different entry.
In \cite{alon}, the authors' study of the hat guessing number of complete bipartite graphs leads them to prove that for any set $S$ of at most $e^{n/q}$ $n$-tuples over $\ZZ_q$, there exists an $n$-tuple $y \in \ZZ_q^n$ such that $d(x,y) < n$ for all $x \in S$ \cite[Lemma 3.3]{alon}.
We will phrase this result in another way.
We say that $x \in \ZZ_q^n$ \underline{skirts} $y \in \ZZ_q^n$ if $d(x,y) = n$, and we say that $S \subseteq \ZZ_q^n$ is a \underline{skirting set} if every $q$-ary $n$-tuple $y$ is skirted by some $x \in S$.
We denote the size of the smallest skirting set in $\ZZ_q^n$ by $f(n,q)$.
Thus, we can interpret \cite[Lemma 3.3]{alon} as stating that $f(n,q) > e^{n/q}$.
In this note, we take a closer look at the function $f(n,q)$.

\bigskip

Note that if $q=2$, then every $n$-tuple only skirts one $n$-tuple, hence trivially $f(n,2)=2^n$, and we will focus on the case $q>2$.

We also remark that if we define the graph $G(n,q)$ with vertex set $\ZZ_q^n$, and where $x$ is adjacent to $y$ if and only if $d(x,y) = n$, then a skirting set $S$ in $\ZZ_q^n$ is exactly the same thing as a \underline{total dominating set} in $G(n,q)$, i.e.\ every vertex of $G(n,q)$ has a neighbor in $S$.

\section{The asymptotics}
We use the asymptotic notation $g(n) = o(h(n))$ to mean that $\lim_{n \to \infty} g(n)/h(n) = 0$.

\begin{thm}\label{thm:rec}
 For every integer $q \geq 2$, there exists a constant $C_q = \inf_{n \geq 1} \sqrt[n]{f(n,q)}$ such that $f(n,q) = C_q^{(1+o(1))n}$.
 Moreover, the sequence $(C_q)_{q\geq 2}$ is non-increasing. 
\end{thm}

\begin{proof}
 Fix an integer $q \geq 2$.
 If $S_n$ and $S_m$ are skirting sets for $\ZZ_q^n$ and $\ZZ_q^m$ respectively of minimum size, then $S_n \times S_m$ (i.e.\ the set of vectors of $\ZZ_q^{n+m}$ obtained by concatenating any vector from $S_n$ and a vector from $S_m$) is a skirting set of $\ZZ_q^{n+m}$.
 It readily follows that $f(n+m,q) \leq f(n,q) f(m,q)$.
 In particular, the sequence $( \ln f(n,q) )_{n \in \mathbb N}$ is subadditive.
 By Fekete's Subadditive Lemma, $\lim_{n \to \infty} \frac{\ln f(n,q)}n$ exists and equals $L_q = \inf_{n \geq 1} \frac{\ln f(n,q)}n$.
 Therefore $\ln f(n,q) = (1+o(1)) n L_q$.
 Define $C_q = e^{L_q} = \inf_{n \geq 1} \sqrt[n]{f(n,q)}$.
 Then we find that $f(n,q) = C_q^{(1+o(1))n}$.

 Now observe that if $S$ is a skirting set for $\ZZ_q^n$, then it is also one for $\ZZ_{q+1}^n$ (where we identify the elements of $\ZZ_q$ with the integers $1, \dots, q$, so that we can see $S$ as a subset of $\ZZ_{q+1}^n$).
 This implies that $f(n,q)$ is non-increasing in $q$, thus, $C_q = \inf_{n \geq 1} \sqrt[n]{f(n,q)}$ is also non-increasing in $q$.
\end{proof}

Hence, to understand the asymptotics of $f(n,q)$, we study the constants $C_q$.
Any upper bound on $f(n,q)$ translates to an upper bound on $C_q$.
To this end, we determine $f(n,q)$ for $n < q$.

\begin{lm}
 If $n < q$, then $f(n,q) = n+1$.
\end{lm}

\begin{proof}
 First we prove that $f(n,q) > n$.
 If $S$ is a set of $n$ vertices in $\ZZ_q^n$, say $x_1, \dots, x_n$, then we can make a vector $y$ with $y(i) = x_i(i)$, which is clearly not skirted by any vector of $S$.

 On the other hand, if $S$ consists of $n+1$ scalar multiples of $\one$, then $S$ is a skirting set, since every $n$-tuple has at most $n$ distinct entries.
\end{proof}

We find the following bounds.

\begin{prop}
 \label{Prop:Bounds}
 For $q \geq 2$, we have $1 + \frac 1{q-1} \leq C_q \leq \sqrt[q-1] q = 1 + \frac{\ln q}{q-1} + o\left( \frac 1q \right)$.
\end{prop}

\begin{proof}
 Since every $q$-ary $n$-tuple skirts $(q-1)^n$ $q$-ary $n$-tuples, if $S$ is a skirting set in $G(n,q)$, then $|S| (q-1)^n \geq q^n$, from which $f(n,q) \geq \left( \frac q{q-1} \right)^n$.
 This implies that $C_q \geq \frac q{q-1}$.
 
 On the other hand, $C_q = \inf_{n \geq 1} \sqrt[n]{f(n,q)} \leq \sqrt[q-1]{f(q-1,q)} = \sqrt[q-1]{q}$.
 Using the Taylor series of the exponential function, we see that $\sqrt[q-1]q = e^{\frac{\ln q}{q-1}} = 1 + \frac{\ln q}{q-1} + \sum_{k=2}^\infty \frac 1{k!} \left( \frac {\ln q}{q-1} \right)^k = 1 + \frac{\ln q} {q-1} + o(1/q)$.
\end{proof}

\section{Small parameters}

\begin{lm}
 If $q \geq 3$, then $f(q,q) \leq 2q-1$.
\end{lm}

\begin{proof}
 Let $e_i$ denote the standard basis vectors in $\ZZ_q^q$.
 Consider the set
 \[
  S = \sett{a \one}{ a = 1, \dots, q-1} \cup \sett{e_i}{i = 1, \dots, q}.
 \]
 We prove that $S$ is a skirting set.
 Take an $n$-tuple $y$.
 Suppose that $y$ is not skirted by any $n$-tuple $a \one$ of $S$.
 Then all symbols from $1$ to $q-1$ occur in $y$.
 Hence, symbol 0 occurs at most once.
 If it occurs in position $i$, then $y$ is skirted by $e_i$.
 If symbol 0 does not occur, then choose a position $i$ with $y(i) = 2$.
 Then $y$ is skirted by $e_i$.
\end{proof}

\begin{prop}
 If $n$ is even, then
 \[
  f(n,3) \leq 2^{n/2} + \sum_{i=1}^{\lceil n/4 \rceil } \binom{n/2}{2i-1} 2^{(n/2) - 2i + 1}.
 \]
\end{prop}

\begin{proof}
Let $n$ be even.
Write $\tilde 0 = 01$, $\tilde 1 = 10$, $\bar 0 = 00$, $\bar 1 = 11$, $\bar 2 = 22$.
Define
\begin{align*}
 S_1 = \left\{ \tilde 0, \tilde 1 \right \}^{\frac n2}, &&
 S_2 = \sett{ x \in \left \{ \bar 0, \bar 1, \bar 2 \right \}^{\frac n2} }{ \text{the number of $i$'s with $x(2i) = 2$ is odd}}.
\end{align*}
We verify that $S = S_1 \cup S_2$ is a skirting set of $\ZZ_3^n$.
Take a vector $y$ in $\ZZ_3^n$, and view $y$ as a vector of $n/2$ blocks of length 2.
If none of the blocks occuring in $y$ are of the form $\bar 0$ or $\bar 1$, then $y$ is skirted by $S_1$.
Indeed, every block that occurs in $y$ either belongs to $\{10, 12, 20, 22 \}$ in which case it is skirted by $\tilde 0$, or it belongs to $\{01, 02, 21 \}$, in which case it is skirted by $\tilde 1$.
Thus, we can make an $n$-tuple $x \in S_1$ out of blocks of the form $\tilde 0$ and $\tilde 1$ that skirts $y$.

Now suppose that $y$ has a block of the form $\bar 0$ or $\bar 1$, say in the $i$th block.
Then we verify that $S_2$ skirts $y$.
Indeed, every block of length 2 contains at most 2 distinct symbols, hence it is skirted by one of the blocks $\bar 0, \bar 1, \bar 2$.
Thus, we can make a vector $x$ consisting of the blocks $\bar 0, \bar 1, \bar 2$ that skirts $y$.
If $x$ has an even number of blocks equal to $\bar 2$, then we can fix this by altering the $i$th block as follows.
Consider first the case where the $i$th block of $y$ is $\bar 0$.
Then the $i$th block of $x$ is either $\bar 1$, in which case we switch it to $\bar 2$, or it is $\bar 2$, in which case we switch it to $\bar 1$.
The resulting alteration of $x$ then has an odd number of occurrences of $\bar 2$, and hence is an element of $S_2$ that skirts $y$.
The case where the $i$th block of $y$ is $\bar 1$ is completely analogous.

We conclude that $S$ is a skirting set of $\ZZ_3^n$.
Next, we compute its size.
 The size of $S_1$ is clearly $2^{n/2}$.
 For a given number $i$, we find $\binom{n/2}{i} \cdot 2^{n/2-i}$
 elements of $\{ \bar{0}, \bar{1}, \bar{2}\}$ 
 with precisely $i$ entries equal to $\bar{2}$.
 Thus, the size of $S_1 \cup S_2$ is as asserted.
\end{proof}

This yields the following upper bounds on $f(n, 3)$:

\begin{center}
\begin{tabular}{c | c c c c c c c}
 $n$ & 4 & 6 & 8 & 10 & 12 & 14 & 16 \\ \hline
 $f(n,3) \leq$ & 8 & 21 & 56 & 153 & 428 & 1221 & 3563
\end{tabular}
\end{center}

Next, we present a table containing some upper and lower bounds on $f(n,q)$ for $q = 3$ or $q=4$ and $n \leq 8$.
The bounds combine our results with some computer generated data.
We present one more helpful tool.

\begin{lm}
 \label{Lm:Increasing}
 The function $f(n,q)$ is strictly increasing in $n$.
\end{lm}

\begin{proof}
 Suppose that $S$ is a minimum size skirting set in $\ZZ_q^n$.
 Consider the function $g: \ZZ_q^n \to \ZZ_q^{n-1}: (x_1, \dots, x_n) \mapsto (x_1, \dots, x_{n-1})$ that removes the last coordinate, and let $S'$ be the image of $S$ under $g$.
 If $y' \in \ZZ_q^{n-1}$, then $y' = g(y)$ for some $y \in \ZZ_q^n$, and $S$ contains some $n$-tuple $x$ that skirts $y$, hence $x' = g(x)$ skirts $y'$.
 Therefore, $S'$ is a skirting set in $\ZZ_q^{n-1}$.
 We claim that $S'$ is not a minimal skirting set.
 Suppose the contrary.
 Take $x = (x',x_n) \in S$ with $x' = g(x) \in \ZZ_q^{n-1}$.
 Then there exists an element $y' \in \ZZ_q^{n-1}$ that is uniquely skirted by $x'$.
 But then the vector $y=(y',x_n)$ is not skirted by $S$, which is a contradiction.
 Hence we may conclude that $f(n-1,q) < |S'| \leq |S| = f(n,q)$.
\end{proof}

If an entry in the table is a single number, then this is the exact value of $f(n,q)$.
Otherwise, the entry is an interval that contains $f(n,q)$.

\begin{table}[h]
 \centering
 \begin{tabular}{c | c c c c c c c}
  $n$ & 2 & 3 & 4 & 5 & 6 & 7 & 8 \\ \hline
  $f(n,3)$ & 3 & 5 & 8 & 12 & 18 & [28,30] & [29,54] \\ 
  $f(n,4)$ & 3 & 4 & 7 & 10 & [13,16] & [14,28] & [15,40]
 \end{tabular}
\end{table}

The table tells us that $C_3 \leq \sqrt[6]{f(6,3)} = \sqrt[6]{18} < 1.6189$, and $C_4 \leq \sqrt[5]{f(5,4)} = \sqrt[5]{10} < 1.5849$. 
The values in the table were found with Gurobi using a standard integer linear program.
Witnesses for $f(5, 3)$, $f(6, 3)$, and $f(5, 4)$ are given below.

\[
 \begin{pmatrix}
  0 & 1 & 1 & 0 & 2 & 2 & 0 & 1 \\
  0 & 1 & 1 & 0 & 2 & 2 & 1 & 0 \\
  0 & 1 & 0 & 1 & 0 & 1 & 2 & 2 \\
  0 & 1 & 0 & 1 & 1 & 0 & 2 & 2
 \end{pmatrix}
\]

\[
 \left( \begin{array}{c c c c c c c c c c c c c c ccccccc}
  0 & 1 & 1 & 0 & 0 & 1 & 1 & 0 & 2 & 2 & 2 & 2 & 0 & 1 & 1 & 0 & 0 & 1 & 1 & 0 & 2 \\
  0 & 1 & 1 & 0 & 0 & 1 & 1 & 0 & 2 & 2 & 2 & 2 & 1 & 0 & 0 & 1 & 1 & 0 & 0 & 1 & 2 \\
  0 & 1 & 0 & 1 & 0 & 1 & 0 & 1 & 0 & 1 & 1 & 0 & 2 & 2 & 2 & 2 & 0 & 1 & 0 & 1 & 2 \\
  0 & 1 & 0 & 1 & 0 & 1 & 0 & 1 & 1 & 0 & 0 & 1 & 2 & 2 & 2 & 2 & 1 & 0 & 1 & 0 & 2 \\
  0 & 1 & 0 & 0 & 1 & 0 & 1 & 1 & 0 & 1 & 0 & 1 & 0 & 1 & 0 & 1 & 2 & 2 & 2 & 2 & 2 \\
  0 & 1 & 0 & 0 & 1 & 0 & 1 & 1 & 1 & 0 & 1 & 0 & 1 & 0 & 1 & 0 & 2 & 2 & 2 & 2 & 2 
 \end{array} \right)
\]

\[
\begin{pmatrix}
    0 & 1 & 0 & 1 & 3 & 3 & 2 & 2 & 1 & 3 \\
    0 & 1 & 0 & 1 & 3 & 3 & 2 & 2 & 3 & 1 \\
    0 & 1 & 2 & 2 & 0 & 1 & 3 & 2 & 3 & 3 \\
    0 & 1 & 2 & 2 & 0 & 1 & 2 & 3 & 3 & 3 \\
    0 & 1 & 1 & 0 & 1 & 0 & 3 & 3 & 2 & 2 
\end{pmatrix}
\]

\section{Skirting arrays}

In this section, we extend the concept of a skirting set to a skirting array.
For an $n$-tuple $x$ and a subset $T \subseteq \{1, \dots, n\}$, we denote by $x_T$ the $|T|$-tuple obtained by only retaining the coordinate positions labelled by elements of $T$.
Say that an $N\times n$ 
array $M$ with entries from $\ZZ_q$ is a  \emph{skirting array of strength $t$}, an $\SA(N;t,n,q)$, if every 
submatrix of $M$ obtained by restriction to $t$ columns skirts the $q$-ary $t$-tuples; i.e.\ for every $t$-set
$T \subset \{1,\ldots,n\}$,  every $y\in \ZZ_q^t$ is at maximum distance from $x_T$ for at least one row $x$ of $M$.
Let $\SAN(t,n,q)$ be the smallest number of rows $N$ in such an array. This is analogous to the notation for 
covering arrays \cite[Section VI.10]{Handbook2006}:  a \emph{covering array of strength $t$} with $N$ rows, 
$n$ columns, and entries from an alphabet of size $q$, denoted a $\CA(N;t,n,q)$, is an array
having the property that restriction to any $t$ columns results in an $N\times t$ array containing every 
possible $t$-tuple at least once as a row. (If the number  of rows containing a $t$-tuple is a constant 
$\lambda$, then this covering array of strength $t$ is, in fact, an orthogonal array of strength $t$.) The 
smallest number of rows $N$ in any $\CA(N;t,n,q)$ is the \emph{covering array number} $\CAN(t,n,q)$. 
Note that $\SAN(t,n,2)=\CAN(t,n,2)$. 

\begin{lm} \label{Lem:CAtoSA}
\[ 
\SAN(t,n,q) \le \CAN(t,n,q) + 1 - (q-1)^t~. 
\]
\end{lm}

\begin{proof}
Let $M$ be obtained from  a $\CA(N;t,n,q)$  by deleting the first $(q-1)^t-1$ rows. If $y$ is any $q$-ary 
$n$-tuple and $T \subset \{1,\ldots,n\}$ is any set of $t$ coordinates, then among the $(q-1)^t$
$t$-tuples at maximum distance from $y_T$, at least one of them remains as the restriction $x_T$
of some row of $M$. 
\end{proof}

\begin{prop}
For any $n$ and $q$, any $1\le v<q$
\[
f(n,q) \le q-v + \SAN(n+v-q,n,v).
\]
\end{prop}

\begin{proof}
Let $M$ be an array composed of an $\SA(N; n+v-q,n,v)$ followed by $\sett{a \one}{ a = v+1, \dots, q}$. If $y$ is any $q$-ary 
$n$-tuple with at least one entry equal to $i$ for each $i>v$, then at most $n+v-q$ entries of $y$ lie in $\{1,\ldots,v\}$
and some row of the $\SA(n+v-q,n,v)$ skirts $y$.
\end{proof}

\begin{crl}
 \label{Crl:SA}
$f(q,q) \le \CAN(t,q,v) - (q-1)^t$.
\end{crl}

Covering array numbers have been studied extensively \cite{tables}. However, the requirement to be a covering array is far
more stringent than what we need, even accounting for the fact that we can reduce the number of tuples by $(q-1)^t-1$.
In Table \ref{Tab:SA}, we present an $\SA(19;4,20,4)$.

\setcounter{MaxMatrixCols}{20}
\begin{table}[h]
\[
A = \begin{bmatrix}
1 & 1 & 1 & 1 & 1 & 1 & 1 & 1 & 1 & 1 & 1 & 2 & 2 & 2 & 3 & 3 & 3 & 4 & 4 & 4 
\\
 1 & 1 & 1 & 1 & 2 & 2 & 2 & 3 & 3 & 3 & 4 & 1 & 1 & 4 & 1 & 3 & 4 & 1 & 3 & 4 
\\
 1 & 3 & 2 & 3 & 4 & 1 & 3 & 1 & 1 & 1 & 4 & 4 & 2 & 4 & 4 & 2 & 1 & 2 & 4 & 3 
\\
 1 & 3 & 4 & 4 & 1 & 1 & 4 & 2 & 2 & 4 & 1 & 3 & 4 & 3 & 1 & 3 & 3 & 1 & 1 & 1 
\\
 2 & 1 & 3 & 4 & 4 & 1 & 4 & 4 & 2 & 3 & 4 & 2 & 1 & 2 & 3 & 2 & 2 & 4 & 2 & 1 
\\
 2 & 2 & 2 & 3 & 2 & 2 & 3 & 3 & 1 & 4 & 4 & 3 & 3 & 3 & 2 & 4 & 3 & 3 & 3 & 2 
\\
 2 & 2 & 3 & 1 & 2 & 4 & 4 & 1 & 4 & 4 & 3 & 3 & 1 & 1 & 1 & 4 & 2 & 2 & 2 & 3 
\\
 2 & 3 & 1 & 4 & 2 & 1 & 4 & 4 & 3 & 1 & 2 & 4 & 4 & 4 & 3 & 2 & 2 & 2 & 2 & 1 
\\
 2 & 3 & 3 & 4 & 3 & 2 & 4 & 2 & 4 & 2 & 3 & 4 & 4 & 2 & 2 & 2 & 1 & 1 & 1 & 2 
\\
 2 & 4 & 4 & 2 & 3 & 3 & 2 & 1 & 1 & 4 & 4 & 4 & 3 & 1 & 1 & 2 & 1 & 1 & 4 & 1 
\\
 3 & 2 & 1 & 1 & 4 & 3 & 4 & 3 & 3 & 3 & 1 & 2 & 2 & 3 & 4 & 3 & 2 & 2 & 3 & 3 
\\
 3 & 2 & 1 & 4 & 3 & 3 & 1 & 2 & 4 & 3 & 3 & 4 & 1 & 4 & 3 & 2 & 4 & 4 & 3 & 4 
\\
 3 & 2 & 2 & 3 & 1 & 2 & 2 & 1 & 2 & 3 & 1 & 1 & 2 & 4 & 3 & 1 & 4 & 1 & 1 & 3 
\\
 3 & 3 & 2 & 1 & 3 & 4 & 1 & 4 & 2 & 2 & 2 & 1 & 2 & 3 & 2 & 4 & 3 & 2 & 2 & 1 
\\
 3 & 4 & 3 & 3 & 4 & 3 & 2 & 1 & 3 & 4 & 4 & 3 & 3 & 2 & 3 & 1 & 4 & 3 & 3 & 1 
\\
 4 & 1 & 1 & 3 & 1 & 4 & 2 & 2 & 4 & 2 & 3 & 2 & 4 & 4 & 4 & 1 & 4 & 2 & 4 & 1 
\\
 4 & 1 & 2 & 2 & 1 & 4 & 3 & 2 & 3 & 1 & 2 & 3 & 1 & 3 & 3 & 3 & 2 & 4 & 2 & 4 
\\
 4 & 2 & 2 & 2 & 4 & 3 & 3 & 4 & 3 & 3 & 2 & 1 & 1 & 2 & 4 & 4 & 3 & 3 & 4 & 2 
\\
 4 & 4 & 1 & 4 & 2 & 1 & 1 & 4 & 2 & 4 & 4 & 2 & 1 & 2 & 3 & 4 & 2 & 2 & 2 & 3 
\end{bmatrix}
\]
 \caption{An $\SA(19;4,20,4)$.\label{Tab:SA}}
\end{table} 

Hence, by Corollary \ref{Crl:SA}, $f(20,20) \leq 35$.

\section{Conclusion}

In this paper, we introduced the concept of skirting sets.
Several problems remain open for future investigation.
\begin{enumerate}
 \item Determine the exact values of the constants $C_q$, or at least improve the bounds presented in Proposition \ref{Prop:Bounds}.
 \item Find constructions of small skirting arrays, and give bounds on $\SAN(t,n,q)$.
 \item In Lemma \ref{Lm:Increasing}, we showed that $f(n,q) > f(n-1,q)$. In the proof, we used that if we take a skirting set in $\ZZ_q^n$, and remove any coordinate, we end up with a skirting set in $\ZZ_q^{n-1}$ with the property that none of its elements are essential, in the sense that after deleting any element, we are still left with a skirting set.
 One would expect that this argument can lead to better lower bounds of $f(n,q) - f(n-1,q)$.
\end{enumerate}



\begin{thebibliography}{XX}
\bibitem{alon} N.~Alon et al., \underline{The hat guessing number of graphs}, J. Combin. Theory Ser. B {\bf 144} (2020), 119--149.
\bibitem{Handbook2006} C.~J.~Colbourn and J.~H.~Dinitz. \underline{Handbook of Combinatorial Designs} (2nd ed.), Chapman \& Hall/CRC, 2007.
\bibitem{tables} Colbourn tables of covering array numbers \url{https://github.com/ugroempi/CAs/blob/main/ColbournTables.md}
\end{thebibliography}
\end{document}